\documentclass[a4paper,11pt,twoside]{article}

\usepackage{sectsty}
\sectionfont{\fontsize{11}{15}\selectfont}

\usepackage{authblk}

\usepackage[margin=1.7in]{geometry}

\usepackage{chicago}      

\usepackage{amsmath}               
\usepackage{amsfonts}              
\usepackage{amsthm}                
\usepackage{amssymb}
\usepackage{amsthm}

\usepackage[pdftex]{hyperref} 

\usepackage{enumerate}

\theoremstyle{plain}
\newtheorem{thm}{Theorem}[section]
\newtheorem{lem}[thm]{Lemma}

\theoremstyle{definition}

\newtheorem{fact}[thm]{Fact}

\DeclareMathOperator{\id}{id}

\newcommand{\PP}{\mathbb{P}}     
\newcommand{\QQ}{\mathbb{Q}}

\newcommand{\on}{{\upharpoonright}}
\newcommand{\forces}{\Vdash}

\newcommand{\incomp}{\perp}

\newcommand{\seq}{\subseteq}

\DeclareMathOperator{\tr}{tr}

\newcommand{\setmin}{{\setminus}}

\DeclareMathOperator{\Borel}{\mathbf{Borel}}

\DeclareMathOperator{\add}{add}

\newcommand{\tle}[1]{2^{< #1}}

\DeclareMathOperator{\nnn}{\mathbf{null}}
 
\usepackage{fancyhdr}
\title{Higher random indestructibility of MAD families}
\author{Thomas Baumhauer\thanks{The author was supported by the Austrian Science Fund through grant FWF P29575.}}
\affil{TU Wien, Institute of Discrete Mathematics and Geometry\\
	Wiedner Hauptstra\ss{}e 8--10, 1040 Wien,
	Austria\\
	\texttt{\href{mailto:thomas.baumhauer@gmail.com}{thomas.baumhauer@gmail.com}}}

\date{}
\begin{document}
	
	\maketitle
	\thispagestyle{empty}
		
	\begin{abstract}
		We give a combinatorial characterization of when
		a maximal almost disjoint family of a weakly compact cardinal $\kappa$
		is indestructible by the higher random forcing $\QQ_\kappa$.
		We then use this characterisation
		to show that $\add(\nnn_\kappa) = \mathfrak b_\kappa = \mathfrak c_\kappa$ implies
		the existence $\QQ_\kappa$-indestructible family.
		The results and proofs presented here are parallel to those for classical random forcing.
	\end{abstract}
	
	\section{Introduction}
	
	In this paper $\kappa$ refers to a weakly compact cardinal.
	A family $\mathcal A \seq [\kappa]^\kappa$ is called {\em almost disjoint}
	if for all distinct $A, B \in \mathcal A$ we have $|A \cap B| < \kappa$.
	An almost disjoint family $\mathcal A$ is called {\em maximal} if for no
	almost disjoint family $\mathcal B \seq [\kappa]^{\kappa}$
	we have $\mathcal A \subsetneq \mathcal B$.
	
	The following way of constructing a maximal almost disjoint family $\mathcal A^*$ of $\kappa$ suggests itself. Identify $\kappa$ with $\tle \kappa$
	and for $\eta \in 2^\kappa$ let $A_\eta = \{\eta \on i : i < \kappa \} \seq \tle \kappa$.	
	Using the Teichm\"uller-Tukey lemma we can extend $\{A_\eta : \eta \in 2^\kappa \}$
	to a maximal almost disjoint family $\mathcal A^*$.
	
	Let $\PP$ be a forcing notion.
	We say that a maximal almost disjoint family $\mathcal A$ is {\em $\PP$-indestructible}
	if $\mathcal A$ remains maximal in any $\PP$-generic extension.
	It is easy to see that any forcing notion $\PP$ adding a real $\eta \in 2^\kappa$ destroys the family $\mathcal A^*$ from above, if $\PP$ satisfies Mostowski's absoluteness.\footnote{
		A forcing notion $\PP$ satisfies Mostowski's absoluteness if
		$\Sigma_1^1$ formulas are absolute between $V$ and $V^\PP$.
		Any $\kappa$-strategically closed forcing has this property,
		see e.g. \cite[Lemma 2.7]{FKK:2016} or \cite[Lemma 4.2.1]{Bh:2019}}
	
	This leads to the question: Given a forcing notion $\PP$, does there exist a maximal almost disjoint family $\mathcal A$
	such that $\mathcal A$ is $\PP$-indestructible?
	For the classical case $\kappa = \omega$ \cite{K:1980} shows that assuming CH there exists 
	a Cohen-indestructible maximal almost disjoint family.
	\cite{H:2001} and \cite{K:2001} provide a combinatorial characterization of Cohen-indestructibility
	and \cite{H:2001} also investigates Sacks and Miller forcing.
	\cite{BY:2005} continue this line of research, investigating several classical
	forcing notions and in particular provide a combinatorial characterization of indestructibility for the classical random forcing.
	
	We shall deal with $\QQ_\kappa$-indestructibility, where
	$\QQ_\kappa$ is the higher random forcing from \cite{Sh:1004}.
	In Theorem \ref{b1} we give a combinatorial characterization of $\QQ_\kappa$-indestructibility,
	parallel to the one in \cite[Theorem 2.4.9.]{BY:2005} for the classical random forcing. In Theorem~\ref{b2}
	we use this characterization to show that 
	\begin{equation}
	\add(\nnn_\kappa) = \mathfrak b_\kappa = \mathfrak c_\kappa
	\label{eqn:1}
	\tag{$*$}
	\end{equation}
	implies
	the existence of a $\QQ_\kappa$-indestructible maximal almost disjoint family of $\kappa$.
	Here $\nnn_\kappa$ denotes the higher null ideal from \cite{Sh:1004}
	(there referred to as $\id(\QQ_\kappa))$ and $\mathfrak c_\kappa$ denotes
	the size of $2^\kappa$.
	This result is again parallel
	to \cite[Theorem 3.6.1.]{BY:2005} where it is shown that $\add(\nnn) = \mathfrak c$ implies
	the existence of a random indestructible maximal almost disjoint family.
	
	Clearly $\mathfrak c_\kappa = \kappa^+$ implies (\ref{eqn:1}).
	However this assumption is not necessary, as the Amoeba model in \cite[Section~6]{BhGoSh:1144} shows (assuming $\kappa$ supercompact)
	that $$\kappa^+ < \add(\nnn_\kappa) = \mathfrak b_\kappa = \mathfrak c_\kappa$$
	is consistent.
	Compared to the classical case we need the additional assumption $\mathfrak b_\kappa = \mathfrak c_\kappa$, as the consistency of $\add(\nnn_\kappa) > \mathfrak b_\kappa$ is
	an open problem.
	
	\section{Notation and Conventions}	
	We use the following conventions.			
	If $f: X \to Y$ is a function, $A \seq X$ and $B \seq Y$, then
	$f[A] = \{f(x) : x \in A \}$ and $f^{-1}[B] = \{x \in X :
	f(x) \in B	\}$.	
	For $\rho \in \tle \kappa$ let
	$[\rho] = \{
	\eta \in 2^\kappa : \rho \trianglelefteq \eta
	\}.$	
	For $\rho, \varrho \in \tle \kappa$ we write $\rho \trianglelefteq \varrho$
	if $\rho \seq \varrho$.
	For $\rho \in p \in \QQ_\kappa$ let $p^{[\rho]} = \{
	\varrho \in p : \varrho \trianglelefteq \rho \lor \rho \trianglelefteq \varrho
	\} \in \QQ_\kappa$.

	For $A \seq \tle \kappa$ define the $G_\delta$-closure
	$$
	[A] = \{
	\eta \in 2^\kappa: \text{there exist cofinally many } i < \kappa \text{ such that }
	\eta \on i \in A
	\} \seq 2^\kappa.
	$$
	(Note that if $A$ is downward closed (i.e. a tree), $[A]$ is a closed set.)
	
	On $2^\kappa$
	we use the topology generated by the basic clopen sets $[\rho]$ for
	$\rho \in \tle \kappa$.
	The $\kappa$-Borel sets $\Borel_\kappa$ are the smallest family containing
	all basic clopen sets which is closed under complements and unions/intersections
	of at most $\kappa$-many sets.
	\section{Higher Random Forcing}
	
	The {\em higher random forcing} $\QQ_\kappa$
	for a (strongly) inaccessible cardinal $\kappa$ was introduced by Saharon Shelah
	in \cite{Sh:1004}.
	Recall that  $\QQ_\kappa$ is a tree forcing on
	$\tle \kappa$ with the following properties:
	
	\begin{enumerate}[$\quad$(a)]
		\item
		$\QQ_\kappa$ satisfies the $\kappa^+$-chain condition.
		\item
		$\QQ_\kappa$ is strategically $\kappa$-closed.
		\item
		If $\kappa$ is weakly compact, then $\QQ_\kappa$ is $\kappa^\kappa$-bounding.
	\end{enumerate}
	
	The {\em higher null ideal} $\nnn_\kappa$ consists of all sets
	$A \seq 2^\kappa$ such that there exists a family
	$\Lambda$ of $\kappa$-many maximal antichains of $\QQ_\kappa$
	such that
	$$
	2^\kappa \setmin A\ \supseteq\ 
	\bigcap_{\mathcal J \in \Lambda }\bigcup_{p \in \mathcal J}[p].
	$$
	
	If $G$ is a $\QQ_\kappa$-generic filter then we call
	$\eta = \bigcup_{p \in G} \tr(p)$ the $\QQ_\kappa$-{\em generic real} or {\em random real},
	where $\tr(p)$ is the trunk of $p$. Throughout the paper $\dot \eta$
	will denote a name for the canonical generic real added by
	$\QQ_\kappa$.
	
	\begin{fact}
		\label{a0}
		Let $p, q \in \QQ_\kappa$. The following are equivalent:
		\begin{enumerate}[~~~~(i)]
			\item 
			$p, q$ are compatible.
			\item 
			$[p] \cap [q] \neq \emptyset$.
			\item 
			$\tr(p) \trianglelefteq \tr(q) \in p \ \lor \ 
			\tr(q) \trianglelefteq \tr(p) \in q$. \qed
		\end{enumerate}
	\end{fact}

	\begin{lem}
		\label{x1}
		Let $p, q \in \QQ_\kappa$. Then:
		\begin{enumerate}[~~~~(i)]
			\item
			If	$q \forces \dot \eta \in [p]$, then  $q \not \incomp p.$
			\item
			If $[q] \seq [p]$, then $q \leq p$
		\end{enumerate}
	\end{lem}

	\begin{proof}
		\ 
		\begin{enumerate}[(i)]
		\item
		Let $q \forces$ ``$\dot \eta \in [p]$'' and towards contradiction assume $q \incomp p$.
		According to Fact~\ref{a0} (iii) there are three cases:
		\begin{enumerate}[~~~~(1)]
			\item 
			$\tr(p) \incomp \tr(q)$
			\item 
			$\tr(p) \trianglelefteq \tr(q) \not \in p$
			\item 
			$\tr(q) \trianglelefteq \tr(p) \not \in q$.
		\end{enumerate}
		As an example consider case (2). For every $\nu \in [p]$ we have $\tr(q) \not \trianglelefteq \nu$. But clearly $q \forces $``$\tr(q) \trianglelefteq \dot \eta$''. Contradiction to $q \forces$ ``$\dot \eta \in [p]$''.
		
	    Work similarly for case (1) and (3).
	    \item
	    Similarly.\qedhere
	    \end{enumerate}
	\end{proof}
	
	\begin{fact}
		\label{a1}
		Let $B \in \Borel_\kappa$. Then:
		\[
		\pushQED{\qed} 
		B \in \nnn_\kappa \quad\Leftrightarrow\quad \forces_{\QQ_\kappa}
		\dot \eta \not \in B.\qedhere
		\popQED
		\]
		This is shown in \cite[Claim 3.2]{Sh:1004} by induction on the Borel rank of $B$.
	\end{fact}

	\begin{fact}
		\label{a7}	 
		 For any $p \in \QQ_\kappa$ we have $[p] \not \in \nnn$.
		 This is a simple consequence of the observation that
		 $p \forces$``$\dot \eta \in [p]$'' and Fact \ref{a1}.\qed
	\end{fact}
	
	\begin{fact}
		\label{a5}
		Let $\kappa$ be weakly compact. If $A \in \nnn_\kappa$,
		then there exists a single maximal antichain $\mathcal J$ of $\QQ_\kappa$
		such that
		$$
		2^\kappa \setmin A\ \supseteq\ 
		\bigcup_{p \in \mathcal J}[p].		
		$$
		This is shown in \cite[Lemma  1.3.3., Lemma 3.1.2]{BhGoSh:1144}.
		\qed
	\end{fact}
	
	\begin{thm}
		\label{a6}
		Let $\kappa$ be weakly compact.
		Let $B \in \Borel_\kappa \setmin \nnn_\kappa$.
		Then there exists $p \in \QQ_\kappa$ such that $[p] \seq B$.
		
		In words: every positive Borel set contains a random condition.
	\end{thm}
	
	\begin{proof}
		By Fact~\ref{a1} there exists $q$ such that $q \forces$``$\dot \eta \in B$''.
		Consider $[q] \setmin B$. There are two cases:
		\begin{enumerate}[~~~~(1)]
			\item
			$[q] \setmin B \in \nnn_\kappa$. By Fact~\ref{a5}
			there exists a single maximal antichain $\mathcal J \seq \QQ_\kappa$
			such that
			$$
			\bigcup_{r \in \mathcal J} [r]\ \cap\ ([q] \setmin B) = \emptyset.
			$$
			Choose $r \in \mathcal J$ compatible with $q$.
			Then $p = r \land q$ is as required.
			\item
			$[q] \setmin B \not \in \nnn_\kappa$.
			By Fact~\ref{a1} there exists $r \in \QQ_\kappa$ such that
			$r \forces \dot \eta \in [q] \setmin B.$
			So in particular
			\begin{enumerate}[(a)]
				\item 
				$r \forces \dot \eta \in [q]$.
				\item
				$r \forces \dot \eta \not \in B$.
			\end{enumerate}
			By (a) and Lemma~\ref{x1}(i) we have
			$r \not \incomp q$. But by our choice of $q$ we have
			 $q \forces$``$\dot \eta \in B$'', hence by (b)
			$q$ and $r$ cannot be compatible. Contradiction, i.e. this case does not appear.\qedhere
		\end{enumerate}
	\end{proof}
	
	\section{Results}
	
	Any maximal almost disjoint family $\mathcal A$ canonically defines
	the ideal $\mathcal I(\mathcal A)$ of all subsets of $\kappa$ that can be
	covered by ${<}\kappa$-many elements of $\mathcal A$.
	Let $\PP$ be a forcing notion.
	We say $\mathcal I(\mathcal A)$ is $\PP$-indestructible if
	$\PP$ does not add a pseudo-intersection to the dual filter of $\mathcal I(\mathcal A)$. Easily
		$\mathcal A$ is $\PP$-indestructible iff
		$\mathcal I(\mathcal A)$ is $\PP$-indestructible.

	\begin{thm}
		\label{b1}
		Let $\kappa$ be a weakly compact cardinal.
		Let $\mathcal A \seq [\kappa]^\kappa$ be a maximal almost disjoint family and let $\mathcal I = \mathcal I(\mathcal A)$.
		The following are equivalent:
		\begin{enumerate}[(i)]
			\item $\mathcal I$ is $\QQ_\kappa$-indestructible.
			\item
			$(\forall B \seq \tle \kappa, [B] \not \in \nnn_\kappa)(\forall f : B \rightarrow \kappa)(\exists I \in \mathcal I)\ [f^{-1}[I] ]\not \in \nnn_\kappa$.
			\item
			$(\forall B \seq \tle \kappa, [B] \not \in \nnn_\kappa)(\forall f : B \rightarrow \kappa,
			f \text{ is } {<}\kappa\text{-to-one})(\exists I \in \mathcal I)$\\ $[f^{-1}[I] ]\not \in \nnn_\kappa$.
		\end{enumerate}

	\end{thm}
	
	\begin{proof}
		\underline{(i) $\Rightarrow$ (ii)}:
		Assume (ii) fails, i.e. there exist $B \seq \tle \kappa$,
		$[B] \not \in \nnn_\kappa$ and
		$f : B \to \kappa$ such that $[f^{-1}[I]] \in \nnn_\kappa$  for all $I \in \mathcal I$.
		By Theorem~\ref{a6} there exists $p \in \QQ_\kappa$ such that $[p] \seq [B]$.
		Let $G$ be a $\QQ_\kappa$-generic filter containing $p$ and let
		$\eta = \bigcup_{q \in G} \tr(q) \in 2^\kappa$, hence by Fact~\ref{a1} 
		we have $\eta \not \in [f^{-1}[I]]$ for all $I \in \mathcal I$.		
		Consider 
		$$A = \{f(\eta \on i) : i < \kappa, \eta \on i \in B \}.$$ First note that
		because $\eta \in [f^{-1}[A]]$ we have $A \not \in \mathcal I$.
		Without loss of generality $\bigcup \mathcal A = \kappa$, hence all sets of size less than $\kappa$ are contained in $\mathcal I$, which implies $|A| = \kappa$.
		
		Now check that $A$ destroys $\mathcal I$.
		Assume it does not, i.e. there exists $I \in \mathcal I$
		such that $|I \cap A| = \kappa$. This implies
		$\eta \on i \in f^{-1}[I]$ for cofinally many $i < \kappa$, hence
		$\eta \in [f^{-1}[I]]$. Contradiction, thus $A$ is almost disjoint from
		all $I \in \mathcal I$, i.e. $p \forces $``$\mathcal I$ is destroyed''.
		So we have shown that $\lnot$(ii) implies $\lnot$(i).
		
		\underline{(ii) $\Rightarrow$ (iii)}: Trivial.

		\underline{(iii) $\Rightarrow$ (i)}:
		Towards contradiction assume there is $p \in \QQ_\kappa$
		and a $\QQ_\kappa$-name $\dot x$ such that
		$$p \forces\text{``}\dot x \in [\kappa]^{\kappa}\text{'' and }
		(\forall I \in \mathcal I)\ 
		p \forces\text{``}|\dot x \cap I| < \kappa\text{''}.$$
		Furthermore let $p$ be a fusion condition
		as in \cite[Claim 1.9.]{Sh:1004}, i.e. such that
		there exists a cofinal sequence $\langle \beta_i : i < \kappa \rangle$ such that for all $i < \kappa$, $\rho \in 2^{\beta_i} \cap p$ the condition
		$p^{[\rho]}$ decides $\dot x_i$, where $\langle \dot x_i : i < \kappa \rangle$
		is an increasing enumeration of $\dot x$.
		
		Let $B = p \cap \bigcup_{i < \kappa} 2^{\beta_i}$
		and clearly $[B] = [p]$, hence $B \not \in \nnn_\kappa$ by Fact~\ref{a7}.
		Define $f : B \to \kappa$
		such that for $\rho \in B \cap 2^{\beta_i}$ we have
		$$f(\rho) = 
		\alpha_\rho \ \text{ such that }\ 
		p^{[\rho]} \forces \text{``}\dot x_i = \alpha_\rho \text{''}.
		$$
		Is is easy to see that $f$ is ${<}\kappa$-to-one since our choice of $p$
		implies $\rho \in B \cap 2^{\beta_i} \Rightarrow f(\rho) \geq i.$
			
		By our assumption there exists $I \in \mathcal I$ such that
		$[f^{-1}[I]] \not \in \nnn_\kappa$, hence by Theorem~\ref{a6} there exists
		$q \in \QQ_\kappa$ such that
		$[q] \seq [f^{-1}[I]]$. Of course $[f^{-1}[I]] \seq [p]$,
		hence $[q] \seq [p]$, and by Lemma~\ref{x1}(ii) this implies
		$q \leq p$.
		
		But $q \forces$``$|\dot x \cap I| = \kappa$''. Contradiction.		
	\end{proof}
	
	Note that the proof of Theorem \ref{b1} essentially verifies that
	$\QQ_\kappa$ satisfies a $\kappa$-version of {\em weak fusion} as defined in \cite[Definition 2.2.1]{BY:2005} (except there a one-to-one function is required).
	However, as the definition of weak fusion is fairly technical, doing the proof
	directly may be more transparent.
	
	\begin{thm}
		\label{b2}
		Let $\kappa$ be a weakly compact cardinal. If $\add(\nnn_\kappa) = \mathfrak b_\kappa = \mathfrak c_\kappa$,
		then there exists a $\QQ_\kappa$-indestructible maximal almost disjoint family
		$\mathcal A \seq [\kappa]^\kappa$ of size $\mathfrak c_\kappa$.
	\end{thm}
	
	\begin{proof}
		Let $\langle (B_\zeta, f_\zeta) : \kappa \leq \zeta < \mathfrak c_\kappa \rangle$
		enumerate all pairs $(B, f)$ where $B \seq \tle \kappa$, $[B] \not \in \nnn_\kappa$ and
		$f : B \to \kappa$ is a ${<}\kappa$-to-one function.
		Let $\langle A_\zeta : \zeta < \kappa \rangle$ be a partition of $\kappa$
		into sets of size $\kappa$.
		We are inductively going to construct sequence $\langle A_\zeta : \kappa \leq \zeta < \mathfrak c_\kappa \rangle$ such that for all $\zeta \in [\kappa, \mathfrak c_\kappa)$:
		\begin{enumerate}[~~~~(1)]
			\item $A_\zeta \in [\kappa]^\kappa$.
			\item 
			$(\forall \epsilon < \zeta)\ |A_\zeta \cap A_\epsilon| < \kappa$.
			\item 
			$(\exists \epsilon \leq \zeta)\ [f_\zeta^{-1}[A_\epsilon]] \not \in \nnn_\kappa$.
		\end{enumerate}	
		If we can carry out this construction, we may find a maximal almost disjoint family
		$\mathcal A  \supseteq \{A_\zeta : \zeta < \mathfrak c_\kappa\}$ using the Teichm\"uller-Tukey lemma, and $\mathcal A$ is $\QQ_\kappa$-indestructible by Theorem \ref{b1}.
	
		At stage $\zeta$ consider $f_\zeta : p_\zeta \to \kappa$.
		
		\underline{Case 1:}
		There exists $\epsilon < \zeta$ such that
		$[f_\zeta^{-1}[A_\epsilon]] \not \in \nnn_\kappa$. 
		In this case let $A_\zeta$ be any set satisfying (1) and (2).
		Remember $\zeta < \mathfrak c_\kappa = \mathfrak b_\kappa \leq \mathfrak a_\kappa$ so this is
		always possible.
		
		\underline{Case 2:}
		For all $\epsilon < \zeta$ we have
		$[f_\zeta^{-1}[A_\epsilon]] \in \nnn_\kappa$.
		By Theorem~\ref{a6} there exists $p_\zeta \in \QQ_\kappa$
		such that $[p_\zeta] \seq [B_\zeta]$.
		By our assumption $\zeta < \add(\nnn_\kappa)$, hence also
		$$X = \bigcup_{\epsilon < \zeta} [f_\zeta^{-1}[A_\epsilon]] \in \nnn_\kappa.$$
		By Fact~\ref{a5} there exists a maximal antichain $\mathcal J$ of $\QQ_\kappa$
		such that
		$$
		X\ \cap\ 
		\bigcup_{p \in \mathcal J}[p]\ =\ \emptyset.
		$$		
		Let $p \in \mathcal J$ be such that $p \not \incomp p_\zeta$
		and let $q = p \cap p_\zeta$. Clearly $X \cap [q] = \emptyset$.
		
		Now the plan is as follows: $f_\zeta[q \cap B_\zeta]$ is a candidate
		for $A_\zeta$ satisfying (1) and (3). So we want to thin out  $f_\zeta[q \cap B_\zeta]$ to some $A_\zeta \seq f_\zeta[q \cap B]$ satisfying (2) and still
		satisfying (1) and (3).		
		We use a combinatorial argument from
		\cite{H:2001} to finish the proof.
		
		Let $\langle \rho_i : i < \kappa \rangle$ enumerate $q \cap B_\zeta$.
		For $i < \kappa$ inductively try to choose distinct $\epsilon_i < \zeta$
		such that $$|f_\zeta^{-1}[A_{\epsilon_i}] \cap q^{[\rho_i]}| = \kappa.$$
		If this construction fails at stage $i < \kappa$ note that
		$$[\bigcup_{j < i} f_\zeta^{-1}[A_{\epsilon_j}] \cap q^{[\rho_i]}] = \emptyset$$
		hence $$[q^{[\rho_i]} \setmin \bigcup_{j < i} f_\zeta^{-1}[A_{\epsilon_j}]]
		= [q^{[\rho_i]}]$$ and easily $$A_\zeta = f[q^{[\rho_i]} \setmin \bigcup_{j < i} f_\zeta^{-1}[A_{\epsilon_j}]]$$ is as required,
		i.e. $A_\zeta$ is almost disjoint from $A_\epsilon$ for all $\epsilon < \zeta$
		and $[f^{-1}_\zeta[A_\zeta]] \supseteq [q^{[\rho_i]}]$, hence
		$[f^{-1}_\zeta[A_\zeta]] \not \in \nnn_\kappa$ by Fact~\ref{a7}.
		
		So assume the construction succeeded and for $\epsilon \in \zeta \setmin \{\epsilon_i : i < \kappa \}$ define $g_\epsilon : \kappa \to \kappa$ by
		$$
		g_\epsilon(i) = \sup(A_\epsilon \cap A_{\epsilon_i}).
		$$
		Remember $\zeta < \mathfrak c_\kappa = \mathfrak b_\kappa$ and
		find $g \in \kappa^\kappa$ such that $g_\epsilon \leq^* g$
		for all $\epsilon$.
		Now for every $i < \kappa$ choose
		$$
		k_i \in \{
		m \in A_{\epsilon_i} : f_\zeta^{-1}[\{m\}] \cap q \cap \{\varrho : \rho_i \trianglelefteq \varrho\} \neq \emptyset
		\ \land\ m > g(i)
		\} \setmin \bigcup_{j < i} A_{\epsilon_j}.
		$$
		Let $A_\zeta = \{k_i : i < \kappa \}$. By construction $A_\zeta$
		is almost disjoint from $A_\epsilon$ for all $\epsilon < \zeta$ and
		$[f_\zeta^{-1}[A_\zeta]] \supseteq [q]$, hence
		$[f_\zeta^{-1}[A_\zeta]] \not \in \nnn_\kappa$ by Fact~\ref{a7}.
	\end{proof}
	
	\section*{Acknowledgements}	
	I thank Martin Goldstern, who read this manuscript and
	provided valuable comments, suggestions and corrections. 
	
	I thank Vera Fischer, who during the defense of my thesis asked the question of the existence of a higher random indestructible maximal almost disjoint family,
	and who thus inspired this paper.
	
	\bibliography{ours}
	\bibliographystyle{chicago}
	
\end{document}